\documentclass[12pt,reqno]{amsart}

\DeclareMathOperator{\codim}{codim}
\DeclareMathOperator*{\E}{{\mathsf E}}

\newcommand{\alp}{\alpha}
\newcommand{\del}{\delta}

\newcommand{\Lam}{\Lambda}

\newcommand{\F}{\mathbb F}

\newcommand{\tf}{{\tilde f}}
\newcommand{\og}{{\bar g}}
\newcommand{\cS}{{\mathcal S}}

\newcommand{\wt}{\widetilde}

\newcommand{\longc}{,\ldots,}
\newcommand{\longe}{=\ldots=}
\newcommand{\longp}{+\dotsb+}

\newcommand{\sub}[1]{{_{\substack{#1}}}}

\newcommand{\seq}{\subseteq}
\newcommand{\stm}{\setminus}

\newtheorem{claim}{Claim}
\newtheorem{lemma}{Lemma}
\newtheorem{theorem}{Theorem}
\newtheorem{corollary}{Corollary}
\newtheorem{proposition}{Proposition}

\newcommand{\refl}[1]{\ref{l:#1}}
\newcommand{\reft}[1]{\ref{t:#1}}
\newcommand{\refc}[1]{\ref{c:#1}}
\newcommand{\refp}[1]{\ref{p:#1}}
\newcommand{\refs}[1]{\ref{s:#1}}
\newcommand{\refb}[1]{\cite{b:#1}}
\newcommand{\refe}[1]{\eqref{e:#1}}

\author{Vsevolod F. Lev}
\address{Department of Mathematics, The University of Haifa at Oranim,
  Tivon 36006, Israel}
\email{seva@math.haifa.ac.il}
\title[Character-free approach to progression-free sets]%
      {Character-free approach \\ to progression-free sets}

\begin{document}
\baselineskip=16pt

\maketitle

\begin{abstract}
We present an elementary combinatorial argument showing that the density of a
progression-free set in a finite $r$-dimensional vector space is $O(1/r)$.
\end{abstract}

\section{Introduction}
A set $A$ of elements of an abelian group is called \emph{progression-free}
if for any $a,b,c\in A$ with $a+c=2b$ one has $a=c$. In \refb{m} Meshulam
proved that for any progression-free subset $A$ of a finite abelian group $G$
of odd order and rank $r\ge 1$, there exist a subgroup $G_0<G$ of rank at
least $r-1$ and a progression-free subset $A_0\seq G_0$ such that the
densities $\alp:=|A|/|G|$ and $\alp_0:=|A_0|/|G_0|$ satisfy
  $$ \alp_0 \ge \frac{\alp-|G|^{-1}}{1-\alp}. $$
Using straightforward induction one easily derives that the density of a
progression-free subset of a finite abelian group of odd order and rank $r$
is $O(1/r)$.

The argument of \refb{m} relies on Fourier analysis, considered now the
standard tool for obtaining estimates of this sort. In this note we introduce
a completely elementary approach, allowing us to establish Meshulam's result
in a purely combinatorial way for the particular case of the additive group
of a finite vector space.

\begin{theorem}\label{t:main}
Let $q$ be a power of an odd prime and $r\ge 1$ an integer. If $A\seq\F_q^r$
is a progression-free set of density $\alp:=q^{-r}|A|$, then there exists a
co-dimension $1$ affine subspace of $\F_q^r$ such that the density of $A$ on
this subspace is at least
  $$ \frac{\alp-q^{-r}}{1-\alp}. $$
\end{theorem}
We remark that if $q,r,A$, and $\alp$ are as in Theorem \reft{main}, and if
$g\in\F_q^r$ and $V<\F_q^r$ is a linear subspace of co-dimension $1$ such
that the affine subspace $g+V$ satisfies the conclusion of the theorem, then
the density of $A$ on $g+V$ is $\alp_0:=|A\cap(g+V)|/|V|$. Hence, if we let
$A_0:=(A-g)\cap V$, then $A_0$ is a progression-free subset of $V$ of density
$\alp_0\ge(\alp-q^{-r})/(1-\alp)$.

We set up the necessary notation in the next section and prove Theorem
\reft{main} in Section \refs{proof}. We admit that our proof is, in a sense,
parallel to that of Meshulam, and therefore may not count as totally new. We
hope, however, that it can be susceptible to various extensions where using
characters is impossible or wasteful.

\section{Preliminaries}

Recall that the averaging operator on a finite non-empty set $S$ is defined
by
  $$ \E_{s\in S}f(s) := \frac1{|S|}\sum_{s\in S} f(s), $$
where $f$ is a real-valued function on $S$. Occasionally, we use
abbreviations as $\E_S(f)$, or just $\E f$, whenever the range is implicit
from the context. By $1_S$ we denote the indicator function of $S$.

For the rest of this section we assume that $f$ is a real-valued function on
a finite abelian group $G$. The $L^2$-norm of $f$ is
  $$ \|f\|_2^2 := \E(f^2), $$
and the balanced part of $f$ is
  $$ \tf := f - \E f, $$
where averaging extends onto the whole group $G$.

Given a linear homogeneous equation $E$ in $k$ variables with integer
coefficients, by $\cS_G(E)$ we denote the solution set of $E$ is
$G^k:=G\times\dotsb\times G$ ($k$ factors), and we let
  $$ \Lam_E[f] := \E_{(g_1\longc g_k)\in\cS_G(E)} f(g_1)\dotsb f(g_k); $$
thus, for instance,
\begin{equation}\label{e:Lam-x-y}
  \Lam_{x-y=0}[f] = \E_{(g,g)\in G\times G} f^2(g) = \|f\|_2^2.
\end{equation}
A basic observation is that if all coefficients of $E$ are co-prime with the
order of the group $G$, then fixing any $k-1$ coordinates of a $k$-tuple in
$G^k$ there is a unique way to choose the remaining coordinate so that the
resulting $k$-tuple falls into $\cS_G(E)$. Consequently, for any real-valued
function $f$ on $G$ and proper non-empty subset $I\subset[k]$ we have
  $$ \E_{(g_1\longc g_k)\in\cS_G(E)} \prod_{i\in I} \tf(g_i) = 0. $$
As a result, if all coefficients of the equation $E$ are co-prime with the
order of the group $G$, then
\begin{equation}\label{e:Lam-alt}
  \Lam_E[\tf] = \Lam_E[f] - (\E f)^k:
\end{equation}
to see this just notice that
\begin{multline*}
  \Lam_E[f] = \E_{(g_1\longc g_k)\in\cS_G(E)} f(g_1)\dotsb f(g_k) \\
       = \sum_{\del_1\longc\del_k=0}^1 (\E f)^{\del_1\longp\del_k}
            \cdot \E_{(g_1\longc g_k)\in\cS_G(E)}
                 \prod_{1\le i\le k \colon \del_i=0} \tf(g_i),
\end{multline*}
and that, by the observation just made, the quantity
  $$ \E_{(g_1\longc g_k)\in\cS_G(E)}
                             \prod_{1\le i\le k \colon \del_i=0} \tf(g_i) $$
vanishes, unless $\del_1\longe\del_k$.

Let $f$ be a real-valued function on a finite abelian group $G$. For a
subgroup $H\le G$ we denote by $f|H$ the function on the quotient group
$G/H$, defined by
  $$ f|H\colon g+H \mapsto \E_{g+H} f. $$
We note that $f|\{0\}=f$, while $f|G=\E f$ is a constant function (on the
trivial group), and that $\E(f|H)=\E f$; furthermore, it is easily verified
that $\tf|H=\widetilde{f|H}$.

\section{Proof of Theorem \reft{main}}\label{s:proof}

We are now ready to prove Theorem \reft{main}. In the heart of our argument
is the identity, established in the following lemma.
\begin{lemma}\label{l:identity}
Let $q$ be a prime power, $r\ge 1$ an integer, and $E$ a homogeneous linear
equation with integer coefficients, co-prime with $q$. If $f$ is a
real-valued function on the vector space $\F_q^r$, then
  $$ \Lam_E[\tf] = \sum_{V<\F_q^r\colon\codim V=1} \Lam_E[\tf|V]. $$
\end{lemma}

\begin{proof}
Denote by $k$ the number of variables in $E$. Writing (for typographical
reasons) $G=\F_q^r$ and agreeing that summation over $V$ extends onto
subspaces $V<\F_q^r$ of co-dimension $1$, we get
\begin{align*}
  \sum_V \Lam_E[\tf|V]
    &= \sum_V \E_{(\og_1\longc\og_k)\in\cS_{G/V}(E)}
                 (\tf|V)(\og_1)\ldots(\tf|V)(\og_k) \\
    &= \sum_V \E_{(\og_1\longc\og_k)\in\cS_{G/V}(E)}
                 \E_{g_1\in\og_1\longc g_k\in\og_k}\tf(g_1)\dotsb\tf(g_k) \\
    &= \sum_V \E_\sub{g_1\longc g_k\in G \\
                          (g_1+V\longc g_k+V)\in\cS_{G/V} (E)}
                                                   \tf(g_1)\dotsb\tf(g_k) \\
    &= \E_{g_1\longc g_k\in G} \tf(g_1)\dotsb\tf(g_k)
         \sum_{V\colon (g_1+V\longc g_k+V)\in\cS_{G/V} (E)} |G/V|.
\end{align*}
We now notice that the number of summands in the inner sum is equal to
$(q^r-1)/(q-1)$ (the total number of co-dimension $1$ subspaces) if
$(g_1\longc g_k)\in\cS_G(E)$, and is equal to $(q^{r-1}-1)/(q-1)$ (the number
of co-dimension $1$ subspaces, containing a fixed non-zero element of $G$) if
$(g_1\longc g_k)\notin\cS_G(E)$. Since $|G/V|=q$ for any subspace $V$ with
$\codim V=1$, we have
  $$ \sum_{V\colon (g_1+V\longc g_k+V)\in\cS_{G/V} (E)} |G/V|
              = \frac{q^r-q}{q-1} + q^r \cdot 1_{\cS_G(E)}(g_1\longc g_k). $$
Hence
  $$ \sum_V \Lam_E[\tf|V]
           = q^r \E_{g_1\longc g_k\in G} \tf(g_1)\dotsb\tf(g_k)
                                        \cdot 1_{\cS_G(E)}(g_1\longc g_k), $$
and it remains to observe that the right-hand side is the definition of
$\Lam_E[\tf]$ in disguise.
\end{proof}

Applying Lemma~\refl{identity} to the equation $x-y=0$ and using
\refe{Lam-x-y}, we obtain
\begin{corollary}\label{c:norm2}
Let $q$ be a prime power and $r\ge 1$ an integer. If $f$ is a real-valued
function on the vector space $\F_q^r$, then
  $$ \|\tf\|_2^2 = \sum_{V<\F_q^r\colon\codim V=1} \|\tf|V\|_2^2. $$
\end{corollary}

Corollary \refc{norm2} can be considered as an analogue of the Parseval
identity. Next, we need a tool for the the ``cubes versus squares''
comparison.

\begin{claim}\label{c:squares-cubes}
Let $M$ be a real number, and $f$ a non-constant real-valued function on the
finite abelian group $G$. Suppose that $E$ is a linear homogeneous equation
in $k\ge 3$ variables with integer coefficients, co-prime with the order of
$G$. If either $k$ is odd and
\begin{align*}
  \Lam_{E}[\tf] &\le -M^{k-2}\|\tf\|_2^2,
\intertext{or $k$ is even and}
  \Lam_{E}[\tf] &\ge M^{k-2}\|\tf\|_2^2,
\end{align*}
then $\max_G\tf\ge M$ and, consequently,
  $$ \max_G f \ge \E f + M. $$
\end{claim}

\begin{proof}
Assuming $\max_G\tf<M$, we get
  $$ \E_{(g_1\longc g_k)\in\cS_G(E)} (\tf(g_1)-\tf(g_2))^2
                                      (M-\tf(g_3))\dotsb(M-\tf(g_k)) > 0. $$
Multiplying out the brackets in the left-hand side and using linearity of the
averaging operator, we get a sum of $3\cdot2^{k-2}$ terms, of which, in view
of the coprimality assumption, not vanishing are only
  $$ 2(-1)^{k-1}\E_{(g_1\longc g_k)\in\cS_G(E)} \tf(g_1)\dotsb \tf(g_k) $$
and
  $$ M^{k-2} \E_{(g_1\longc g_k)\in\cS_G(E)}
                                          (\tf(g_i))^2;\quad i\in\{1,2\}. $$
Of these three terms the first is equal to $2(-1)^{k-1}\Lam_E[\tf]$, while
coprimality ensures that the other two are both equal to
$M^{k-2}\|\tf\|_2^2$. We conclude that
  $$ (-1)^k\Lam_E[\tf] < M^{k-2}\|\tf\|_2^2, $$
contradicting the assumptions.
\end{proof}

We notice that Claim~\refc{squares-cubes} is, in a sense, sharp: say, if
$G,E$, and $k$ are as in the claim, and the coefficients of $E$ add up to
$0$, then for any fixed element $g\in G$, writing $N:=|G|$ we have
  $$ \Lam_E[\wt1_{G\stm\{g\}}] = (-1)^kN^{-(k-2)} \|\wt1_{G\stm\{g\}}\|_2^2 $$
(both sides being equal to $(-1)^kN^{-k}(N-1)$), whereas
$\max_G\wt1_{G\stm\{g\}}=N^{-1}$.

Suppose that $q,r,E$, and $f$ are as in Lemma~\refl{identity}. Comparing
\begin{align*}
  \sum_{V<\F_q^r\colon\codim V=1} \Lam_E[\tf|V] &= \Lam_E[\tf]
\intertext{(which is the conclusion of the lemma) and}
  \sum_{V<\F_q^r\colon\codim V=1} \|\tf|V\|_2^2 &= \|\tf\|_2^2
\end{align*}
(by Corollary~\refc{norm2}), and observing that $\|\tf|V\|_2=0$ implies
$\Lam_E[\tf|V]=0$, we conclude that if $f$ is not a constant function, then
there exists a co-dimension $1$ subspace $V<\F_q^r$ with $\|\tf|V\|_2\ne0$
and
  $$ \Lam_E[\tf|V] \le \frac{\Lam_E[\tf]}{\|\tf\|_2^2}\cdot \|\tf|V\|_2^2. $$
If $E$ is an equation in three variables, then by Claim~\refc{squares-cubes}
there exists $g\in\F_q^r$ such that
  $$ \E_{g+V} f = (f|V)(g) \ge \E f - \frac{\Lam_E[\tf]}{\|\tf\|_2^2}. $$
We summarize as follows.
\begin{proposition}\label{p:density-increment}
Let $q$ be a prime power, $r\ge 1$ an integer, and $E$ a homogeneous linear
equation in three variables with the coefficients, co-prime with $q$. If $f$
is a non-constant real function on the vector space $\F_q^r$, then there
exists a co-dimension $1$ affine subspace $g+V$ (where $g\in\F_q^r$ and
$V<\F_q^r$ is a linear subspace) such that
  $$ \E_{g+V} f \ge \E f - \frac{\Lam_E[\tf]}{\|\tf\|_2^2}. $$
\end{proposition}

We are now in a position to complete the proof of Theorem \reft{main}, and
this is where the property of being progression-free comes into play.

Suppose that $G$ is a finite abelian group and $A\seq G$ is a subset of
density $\alp:=\E1_A$. By \refe{Lam-x-y} and \refe{Lam-alt}, we have
  $$ \|\wt1_A\|_2^2 = \Lam_{x-y=0}[\wt1_A] = \E1_A^2 - (\E1_A)^2
                                                            = \alp(1-\alp). $$
If $N:=|G|$ is odd and $A$ is progression-free, then
$\Lam_{x-2y+z=0}[1_A]=N^{-1}\alp$ by the definition of the operator $\Lam_E$;
hence \refe{Lam-alt} gives
  $$ \Lam_{x-2y+z=0}[\wt1_A] = -\alp(\alp^2-N^{-1}). $$
Thus,
  $$ \E1_A - \frac{\Lam_{x-2y+z}[\wt1_A]}{\|\wt1_A\|_2^2} =
         \alp + \frac{\alp^2-N^{-1}}{1-\alp} = \frac{\alp-N^{-1}}{1-\alp} $$
and Theorem \reft{main} follows from this equality and Proposition
\refp{density-increment}.

\end{document}